\documentclass[10pt]{article}
\usepackage{amssymb,latexsym,amsmath,epsfig,amsthm}
\usepackage{graphicx,float,hyper ref,listings}

\makeatletter

\renewcommand\section{\@startsection {section}{1}{\z@}
{-30pt \@plus -1ex \@minus -.2ex}
{2.3ex \@plus.2ex}
{\normalfont\normalsize\bfseries\boldmath}}

\renewcommand\subsection{\@startsection{subsection}{2}{\z@}
{-3.25ex\@plus -1ex \@minus -.2ex}
{1.5ex \@plus .2ex}
{\normalfont\normalsize\bfseries\boldmath}}

\renewcommand{\@seccntformat}[1]{\csname the#1\endcsname. }

\makeatother

\theoremstyle{theorem}
\newtheorem{theorem}{Theorem}
\newtheorem{lemma}[theorem]{Lemma}
\newtheorem{proposition}[theorem]{Proposition}
\newtheorem{corollary}[theorem]{Corollary}

\theoremstyle{definition}
\newtheorem{definition}{Definition}

\newtheorem{remark}{Remark}

\begin{document}

\title{The Integer Sequence Transform $a \mapsto b$ where $b_n$ is the Number of Real Roots of the Polynomial  $a_0 + a_1x + a_2x^2 + \cdots + a_nx^n$}

\author{W.~Edwin Clark\\Department of Mathematics\\University of South Florida \\ 33620 Tampa, FL, USA\\wclark@mail.usf.edu\\
~\\Mark Shattuck\\ Department of Mathematics \\University of Tennessee\\
37996 Knoxville, TN, USA\\shattuck@math.utk.edu}

\date{\today}

\maketitle

\begin{abstract}
\noindent We discuss the integer sequence transform $a \mapsto b$ where $b_n$ is the number of real roots of  the polynomial $a_0 + a_1x + a_2x^2 + \cdots + a_nx^n$.  It is shown that several sequences $a$ give the trivial sequence $b = (0,1,0,1, 0,1,\ldots)$, i.e., ${b_n = n \bmod 2}$, among them the Catalan numbers, central binomial coefficients, $n!$ and $\binom{n+k}{n}$ for a fixed $k$.  We also look at some sequences $a$ for which $b$ is more interesting such as $a_n = (n+1)^k$ for $k \geq 3$.  Further, general procedures are given for constructing real sequences $a_n$ for which $b_n$ is either always maximal or minimal.
\end{abstract}

\small{

} \normalsize

\section{Introduction} In this paper $a = (a_0,a_1,\ldots,a_n,\ldots)$ will, unless otherwise indicated, denote an integer sequence  with $a_0 \neq 0$.
Given such a sequence $a$ of integers, define the sequence $RR(a) = b$ where $b_n$ is the number of real roots  counting multiplicities of  the polynomial $\sum_{i=0}^{n} a_ix^i$.
Some obvious properties of $b$ are
\begin{description}
    \item [(1)\label{itm:(1)}] $b_0 = 0.$
    \item [(2)\label{itm:(2)}] $b_1 = 1$ if $a_1 \neq 0$.
    \item [(3)\label{itm:(3)}] $0 \leq b_n \leq n$ if $n$ is even.
    \item [(4)\label{itm:(4)}] $1 \leq b_n  \leq n$  if $n$ is odd and $a_n \neq 0$.
    \item [(5)\label{itm:(5)}] if $a_n = 0$, then $b_n = b_{n-1}.$
    \item [(6)\label{itm:(6)}] $b_n \equiv n \pmod{2}$ if $a_n\neq 0$.
\end{description}

For many integer sequences  $a$, the sequence $b$ is the minimal sequence given by $b_n = 0$ for $n$ even and $b_n = 1$ for $n$ odd, i.e., ${b_n = n \bmod 2}$ (see  Section~\ref{someseq}).
On the other hand, it is impossible to have $b_n = n$ for all $n$ for any integer sequence $a$.  As pointed out to us by Mourad Ismail, it turns out that some related questions have been considered in analysis, e.g., \cite{Ganelius},\cite{KLV},\cite{Korevaar} and \cite{O}. For example, let
$$\sum_{k=0}^{\infty}a_kz^k $$
be a formal power series. The polynomial
$$ S_n(z):=\sum_{k=0}^{n}a_kz^k $$
is called its $n$-th {\it section}. According to \cite{O}, the following theorem goes back to Polya in 1913.
\begin{theorem} (See \cite{KLV} and \cite{O}.) If for the formal power series
\[\sum_{k=0}^{\infty}a_kz^k , \hspace{2mm} a_k > 0, \tag{*} \label{eq:series} \]
all of the roots of the sections $S_n(z)$ are real for all sufficiently large $n$, then the radius of convergence of the power series (*) is infinite.
\end{theorem}

A stronger result without the assumption that the $a_k$ in $\sum_{k=0}^{\infty}a_kz^k $ are real is given in the following theorem.
\begin{theorem}\label{Ganelius}  (See Ganelius \cite{Ganelius}.)  Let $f_n(z)=\sum_{k=0}^{n}a_kz^k$ be the partial sums of the power series  $f(z)=\sum_{k=0}^{\infty}a_kz^k $, $a_0 = 1$. Let $N(S,n)$ be the number of zeros of the polynomial $f_n(z)$  in a sector $S  = \{z : \,  \beta  \leq arg(z) \leq \alpha\ + \beta \}$ for some $\beta$ and some $\alpha > 0$ and assume $N(S,n)=o(n)$, then  $f(z)$ is an entire function of order zero.
\end{theorem}

If all zeros  of the $n$-sections are on the real line, then obviously there are many sectors which have no zeros at all. Hence, from Theorem~\ref{Ganelius} we have the  following.

\begin{corollary} Under the conditions of Theorem~\ref{Ganelius}, $\lim_{n \to \infty} a_n = 0$. Thus, if the  $a_n$ are integers, $f(z)$ must be a polynomial and $b_n = n$ for all $n$ is not possible.
\end{corollary}
\remark{}However $b_n = n$ for all $n$ is possible if we allow the $a_n$ to  be rational  numbers. For  example, from \cite{KLV}, if we take $c^2 \geq 4$, then
 $$\sum_{k=0}^{n} c^{-k^2}z^k$$
has only real roots for all $n$.

Lest one is left with the impression that the considerations of this paper are solely in the province of analysis, we point out that Stanley \cite{RPS} has considered the combinatorial significance of having all real roots for a polynomial $$a_0 + a_1x + a_2x^2 + \cdots + a_nx^n, \hspace{2mm}  a_i \in  \mathbb{N}.$$
See also Liu and Wang \cite{Wang} for applications to combinatorics.

The organization of this paper is as follows.  In Section~\ref{Computations}, we note that the computation of $RR(a)$ for integer sequences can be accomplished by exact rational number arithmetic.  In Section~\ref{someseq}, we consider several examples where $a_n$ gives rise to the minimal sequence $b_n$. In \ref{kpower}, we conjecture based on numerical evidence that sequence {\bf A346379} in \cite{OEIS}, which we recently authored, is defined for all $n \geq 0$.  This would provide examples of sequences $a_n$ for which $b_n$ starts $0,1,\ldots,m$ for any $m \geq 0$.  In \ref{completereal}, we introduce the notion of a completely real polynomial, which is in a sense the antithesis of the concept considered in \ref{someseq}, and look at some examples.  Further, we provide a way of constructing completely real polynomials of arbitrary degree.   In Section~\ref{proofs}, we prove that ${b_n = n \bmod 2}$ for a general class of sequences and obtain as special cases the minimality of $b_n$ for such well-known sequences as the Catalan numbers, central binomial coefficients and $n!$.  We also establish necessary and sufficient conditions for the minimality of $b_n$ in the case when $a_n$ is a quadratic polynomial in $n$ with arbitrary coefficients and find a general procedure of constructing $a_n$ for which ${b_n = n \bmod 2}$.  Finally, we provide a criterion for determining when an explicitly defined sequence of positive real numbers $a_n$ yields a minimal $b_n$, which can be used to show ${b_n = n \bmod 2}$ in such cases as $a_n=n^n$ and $a_n=2^{n^k}$, where $k>1$ is fixed.

 \section{Possible Generalizations} In this paper, we define $RR(a) = b$ where $b_n$ is the number of real zeros of the $n$-th section $\sum_{i=0}^n a_ix^i$ of the ordinary generating function $G(x) = \sum_{i=0}^{\infty} a_ix^i$ of the sequence $a$. We could just as well  have replaced the ordinary generating function by some other type of generating function, e.g., exponential generating function, Lambert series, Bell series or Dirichlet series, see \cite{G}.  We have not investigated these other generating functions.

 One may also  define the {\it dual} sequence $\tilde{b}_n = n -b_n$, the number of non-real  roots of $\sum_{i=0}^n a_i x^i$. So, for example, the dual  of $b_n = n  \bmod 2$ is the sequence
$$
0, 0, 2, 2, 4, 4, 6, 6, 8, 8, 10, 10, 12, 12, 14, 14, 16, 16, 18, 18, 20,20,22,24,24,\dots
$$

\section{Computational  Considerations} \label{Computations} Given a polynomial $p \in \mathbb{Z}[x]$, the number of real roots may be computed using exact rational number arithmetic. We briefly describe this here. For more details, see \cite{RRI} and \cite{KRS}.  In order to count the number of real roots considering multiplicity, it is convenient to first find the square free decomposition of $p$:  $$p = uf_1^{e_1}f_2 ^{e_2}\cdots f_k^{e_k},$$ where $f_i \in \mathbb{Z}[x]$,  $e_i \in \mathbb{N}$, $|u|$ is  the greatest common divisor of the coefficients of $p$, each $f_i$ has no multiple roots,  i.e., $\gcd(f_i,f_i^\prime)= 1$, and when $i \neq j$, $\gcd(f_i,f_j) = 1$. Maple's procedure {\bf sqrfree} produces this decomposition.
Next, we use a procedure (in Maple, {\bf realroot}) which produces for each $f_i$ a list of pairwise disjoint intervals each containing exactly one root of $f_i$. If $N_i$ is the number of such intervals in the list for $f_i$, then $\sum_{i=1}^{k} N_i e_i$ is the number of real roots of $p$ counting multiplicities of the roots.
Our Maple procedure, {\bf NumRealRoots}, for computing the number of real roots including multiplicities for a polynomial $p \in \mathbb{Z}[x]$ is the following:

\begin{verbatim}
NumRealRoots := proc(p)
local q, k, u;
      if p = 0 then error "zero polynomial not allowed"; fi;
      q := sqrfree(p);
      k := 0;
      for u in q[2] do
          k := k+nops(realroot(u[1]))*u[2]
       end do;
      k
end proc
\end{verbatim}

We note that in {\bf NumRealRoots}, we could have substituted for $${\bf nops(realroot(u[1]))}$$ the Maple procedure based on Sturm's Theorem,
$${\bf sturm(sturmseq(u[1], x), x, -infinity, infinity)},$$
 but this produces a slower procedure.

\section{Motivating Example \label{MotEx}} This paper was motivated by the following problem raised on the Mathematics StackExchange by {\it Russian Bot 2.0} \cite{MSE}. Let $p_n$ denote the $n$th prime for $n \geq 1$ and define the sequence
$$a_n= p_{n+1}, \qquad n  \geq 0.$$
Consider the sequence $b$ where $b_n$ is the number of real roots (counting multiplicities) of the polynomial
$$P_n(x) =2+3x +5x^2 + 7x^3 + \dots + a_nx^n.$$
The sequence $b$ starts out
\begin{equation}
0,1,0,1,0,1,0,1,0,1,0,1,0,1,0,1,0,1,0,1,0,1,0,1,0,1,0,1,\dots
\end{equation}
That is  $b_n = n \bmod 2$.
The question was asked whether or not this pattern continues. It was pointed out that the pattern is broken for the first time for $n=2436$ when $b_{2436} = 2$, and thereafter at least to  $n = 2730$,  $b_n = 2$ for $n$ even and $b_n = 1$ for $n$ odd. How far this new pattern continues is unknown. However, jumping ahead we find that $b_{4000} =  2$,  $b_{4001} = 1$, $b_{5000} = 2$, $b_{5001} =1$, $b_{10000}=2$,  $b_{10001} = 1$, $b_{20000}=2$ and $b_{20001} = 1$.

An anonymous contributor to the answer to this problem, namely, {\it user3733558}, mentioned that the pattern,  $b_n = n \bmod 2$,  seems to hold  at least initially for several other sequences. We consider some examples of such sequences in the next section.

\section{Some Sequences for which $b = RR(a)$ is ${b_n =  n \bmod 2}$\label{someseq}} Here, we give several examples of sequences $a_n$ for which $b_n = n \bmod 2$, where proofs are provided of this result in Section \ref{proofs} below for the sequences in \ref{binom}--\ref{5.7}.

\subsection {Fibonacci Numbers} Let $(0,1,1,2,3,5,\ldots )$ be the sequence of Fibonacci numbers $F_n$. Then we define $a_{n} = F_{n+1}$ for $n \geq 0$. If $b = RR(a)$, it follows  from the results of Garth, Mills and Mitchell \cite{GMM} that $b_n = 0$ for even $n$ and  $b_n = 1$ for odd $n$. We note that the polynomials $p(x)$ of degree $n$  they consider are the reciprocal polynomials $x^nS_n(1/x)$ of  the $n$-th sections $S_n(x)$ of the generating function of $a$, but this  has  no effect on the  number of  real roots as $a_0\neq 0$. Thus ${b_n = n \bmod 2}$ for $n \geq 0$.

\subsection{$k$-Fibonacci Numbers} As in Mansour and Shattuck \cite{MS}, let the recursive sequence ${a_n}$, $n \geq 0$, of order $k,$  $k \geq 2$,  be defined by the initial
values $a_0 = a_1 = \cdots= a_{k-2} = 0$, $a_{k-1} = 1$ and the linear recursion
$$a_n = a_{n-1 }+ \cdots + a_{n-k},  \quad n \geq k.$$
Mansour and Shattuck define, for $k \geq 2$ and $n \geq 1$, the polynomial
$$P_{n,k}(x) = a_{k-1}x^n + a_kx^{n-1}+\cdots + a_{n+k-2}x + a_{n+k-1},$$
and they prove for each $k \geq 2$ that the polynomial $P_{n,k}$ has zero real roots if $n$ is even and one real root if $n$ is odd.
As in the paper of Garth, Mills and Mitchell, these are the reciprocal polynomials of the $n$-th sections  of the  generating  function
$\sum_{i = 0}^{\infty} a_{i+k-1}x^i$ of  the sequence that is of interest to us. But again this has no effect on the number of real roots. So if we begin the sequence at $n=k-1$, we find that ${b_n = n \bmod 2.}$

\subsection {Exponential Sequences $a_n = c^n, c \in \mathbb{R}$ and $c \neq 0$ \label{expon}} Note that this includes the sequence $a_n = 1$ for all $n$. We seek  the number of real zeros of the polynomials
$$ p_n(x):= 1 + (cx)^1 + (cx)^2 + \cdots + (cx)^n = \frac {(cx)^{n+1} - 1}{cx-1}.$$
Note that the polynomial $x^{n+1} - 1$ has only one real root if $n+1$ is odd, namely,  $1$, and two  real roots  $1$ and $-1$ if $n+1$ is  even.
It follows that $p_n(x)$ has zero real  roots if $n$ is  even and one real  root, namely, $x = -1/c$, if $n$ is odd, that is, ${b_n = n \bmod 2}.$

\subsection{Binomial Sequences: For fixed $k \in \mathbb{N}$, $a_n={ {n+k}\choose{k}}$  \label{binom}}
Computations show that ${b_n = n \bmod 2}$ for $k$ from $1$ to $20$ and $n$ from $0$ to $500$.  We thus seek  the number of real zeros $b_n$ of the polynomials
\begin{align*}
p_n(x) &= \sum_{i=0}^n \binom{i+k}{k}x^i\\
&=\frac{1}{(1-x)^{k +1}}-\binom{n +1+k}{k} x^{n +1} {}_{2}F_{1} ( 1,n +2+k ;n +2;x ),
\end{align*}
where the latter expression is given by Maple.

\subsection{$a_n = {n+1}$}\label{5.5}  This corresponds to the $k=1$ case of $\binom{n+k}{k}$ and $(n+1)^k$ from sections \ref{binom} and \ref{kpower}. Computations show that $b = RR(a)$ is the minimal sequence for all $n$ up to $1000$.
Thus, we need to consider the real zeros of the polynomials
 $$  p_n(x) = \frac{1+x^{n +1} \left(\left(n +1\right) x -n -2\right)}{\left(x -1\right)^{2}}.$$

\subsection{$a_n = {(n+1)}^2$}\label{5.6}  This is the case $k=2$ of $(n+1)^k$ from Section~\ref{kpower}. Computations suggest that $b=RR(a)$ is trivial. Thus, we seek  the number of real zeros $b_n$ of the polynomials
\begin{align*}
p_n(x) &= \sum_{i=0}^n {(i+1)^2}x^i\\
&=\frac{\left(n +1\right)^{2} x^{n+3}-\left(2 n^{2}+6 n +3\right)x^{n+2}+\left(n +2\right)^{2}x^{n+1}-x -1}{\left(x -1\right)^{3}}.
\end{align*}

\subsection{Catalan Numbers: $a_n = \frac{1}{n+1}\binom{2n}{n}$}\label{5.7}  Again, computations suggest ${b_n =  n \bmod 2}$ for all $n$. We remark that in \cite{Sh}, it was proven that the Motzkin, Schr\"{o}der and some related sequences each give rise to the minimal $b_n$ sequence.

\section{$k$-th Power Sequences $a_n = (n+1)^k$ for $k \geq 1$  \label{kpower}} Consider a typical sequence of this type: Let $k=50$, then  for $a_n = (n+1)^{50}$, the sequence  $b=RR(a)$ begins:
$$
\begin{array}{cccccccccccccccccccc}
0 ,& 1 ,& 2 ,& 3 ,& 4 ,& 5 ,& 6 ,& 5 ,& 6 ,& 7 ,& 6 ,& 7 ,& 8 ,& 9 ,& 8
\\
 9 ,& 8 ,& 9 ,& 10 ,& 9 ,& 10 ,& 11 ,& 10 ,& 11 ,& 10 ,& 11 ,& 12 ,& 11 ,& 12 ,& 11
\\
 12 ,& 13 ,& 12 ,& 13 ,& 12 ,& 13 ,& 14 ,& 13 ,& 14 ,& 13 ,& 14 ,& 13 ,& 14 ,& 13 ,& 14
\\
 15 ,& 14 ,& 15 ,& 14 ,& 15 ,& 14 ,& 15 ,& 14 ,& 15 ,& 16 ,& 15 ,& 16 ,& 15 ,& 16 ,& 15
\\
 16 ,& 15 ,& 16 ,& 15 ,& 16 ,& 17 ,& 16 ,& 17 ,& 16 ,& 17 ,& 16 ,& 17 ,& 16 ,& 17 ,& 16
\\
 17 ,& 16 ,& 17 ,& 16 ,& 17 ,& 16 ,& 17 ,& 18 ,& 17 ,& 18 ,& 17 ,& 18 ,& 17 ,& 18 ,& 17
\\
 18 ,& 17 ,& 18 ,& 17 ,& 18 ,& 17 ,& 18 ,& 17 ,& 18 ,& 17 ,& 18 ,& 19 ,& 18 ,& 19 ,& 18
\\
 19 ,& 18 ,& 19 ,& 18 ,& 19 ,& 18 ,& 19 ,& 18 ,& 19 ,& 18 ,& 19 ,& 18 ,& 19 ,& 18 ,& 19
\\
 18 ,& 19 ,& 18 ,& 19 ,& 18 ,& 19 ,& 18 ,& 19 ,& 18 ,& 19 ,& 20 ,& 19 ,& 20 ,& 19 ,& 20
\\
 19 ,& 20 ,& 19 ,& 20 ,& 19 ,& 20 ,& 19 ,& 20 ,& 19 ,& 20 ,& 19 ,& 20 ,& 19 ,& 20 ,& \cdots
\end{array}
$$
Note that $b_n = n$ for $n = 0,1,2,3,4,5,6$.

Computations suggest that for all $n$, there is a $k$ such that if $a_n = (n+1)^k$, then the sequence $b = RR(a)$ satisfies
$b_i = i$ for $i = 0, 1, \ldots,n$. For the smallest such $k$, we write   $c_n = k$.
 The values for $c_n$  up to $n = 27$ are given by Table~\ref{table:1}.
 See the new sequence {\bf A346379} in the OEIS \cite{OEIS} for the most recent values.  We conjecture that $c_n$ is indeed defined for all $n \geq 0$.
A Maple calculation shows that for $n = 125$ and $k = 20000$ the number of real roots of $\sum_{i=0}^j (i+1)^kx^i$ is $j$ for $j = 0,1,\dots,n$. This shows that $c_n$ is defined at least up to $n = 125$.

\begin{table}
\begin{center}
\begin{tabular}{ |c|c||c| c||c|c|c|c|}
 \hline
n & $c_n$ &n&$c_n$&n&$c_n$&n&$c_n$\\
\hline
0 & 0 & 7 & 63 & 14 & 242 & 21 & 536\\
1 & 0 & 8 & 81 & 15 & 277 & 22 & 587\\
2 & 5 & 9 & 102 & 16 & 314 & 23 & 641\\
3 & 12 & 10 & 125 & 17 & 354 & 24 & 697\\
4 & 21 & 11 & 151 & 18 & 396 & 25 & 755\\
5 & 33 & 12 & 179 & 19 & 440 & 26 & 816\\
6 & 47 & 13 & 209 & 20 & 487 & 27 & 879\\
 \hline
\end{tabular}
\end{center}
\caption{}
\label{table:1}
\end{table}

For example, since $c_{10} = 125$, the sequence $b$ corresponding to $a_i = (i+1)^{125}$ begins:
$$
\begin{array}{ccccccccccccccccc}
0 & 1 & 2 & 3 & 4 & 5 & 6 & 7 & 8 & 9 & 10 & 9 & 10 & 11 & 12 & 11 & 12
\\
 13 & 14 & 13 & 14 & 15 & 16 & 15 & 16 & 17 & 18 & 17 & 18 & 19 & 18 & 19 & 20 & 19
\\
 20 & 21 & 20 & 21 & 22 & 21 & 22 & 23 & 22 & 23 & 24 & 23 & 24 & 23 & 24 & 25 & 24
\\
 25 & 26 & 25 & 26 & 27 & 26 & 27 & 26 & 27 & 28 & 27 & 28 & 27 & 28 & 29 & 28 & 29
\\
 28 & 29 & 30 & 29 & 30 & 29 & 30 & 31 & 30 & 31 & 30 & 31 & 30 & 31 & 32 & 31 & 32
\end{array}
$$

\section{Completely Real Polynomials \label{completereal}}

\begin{definition}
We shall say that a polynomial $p = \sum_{i=0}^na_ix^i  \in \mathbb{R}[x]$ is {\bf completely real} if all sections $p_i= \sum_{j=0}^ia_jx^j$ for $0 \leq i \leq n$ have $i$ real roots. As in the case of  formal power series, we write $b_i$ for  the number of  real roots  of  $p_i.$
\end{definition}

In this section, we exhibit some  completely real  polynomials in $\mathbb{Z}[x]$. The first example  shows that for all natural numbers $n$ there exist completely real polynomials in $\mathbb{Z}[x]$ of degree $n$.

\remark{} While the reciprocal  polynomial $p^*(x) = x^np(1/x)$ of a completely real polynomial $p(x)$ is often completely real, this is not always the case. For example, the polynomial $p = x^{3}+99 x^{2}+2456 x +13827$ is completely real, but its reciprocal $p^* = 13827 x^{3}+2456 x^{2}+99 x +1$ is not.

\subsection{$\sum_{k=0}^{n} c^{n^2-k^2}z^k$, $c^2 \geq 4$}
As pointed out in  \cite{KLV},  if $c^2 \geq 4$, then for all $i \geq 0$ the polynomials
 $$\sum_{k=0}^{i} c^{-k^2}z^k$$
have only real roots.
Clearly then for $n \in \mathbb{N}$, the polynomials
$$c^{n^2} \sum_{k=0}^{i} c^{-k^2}z^k = \sum_{k=0}^{i} c^{n^2-k^2}z^k$$
have only real roots. The latter are the $i$-sections of the polynomial
$$\sum_{k=0}^{n} c^{n^2-k^2}z^k,$$
which is therefore completely real and in $\mathbb{Z}[x]$ if
$c \in \mathbb{Z}.$

\subsection{$\sum_{i=0}^n (i+1)^kx^i$ if $c_n = k$} From Section~\ref{kpower}, it follows that if $c_n = k$, then the polynomials
$$\sum_{i=0}^n (i+1)^kx^i$$
are   completely real.

\subsection{$p={\prod}_{i=1}^{n} \left(i^{k} x +1\right)$, $k \in \mathbb{N}$}
These polynomials in $\mathbb{Z}[x]$ are designed to have $n$ real roots $-i^{-k}$, $i = 1,2, \ldots, n$. It turns out that for the pairs $(n,k)$ in Table~\ref{table:2}, computations show that $b_i = i$ for  $i = 0,1, \ldots,n$. In a pair $(n,k)$, $k$ is the smallest number that works for $n$.
\begin{table}
\begin{center}
\begin{tabular}{ |c|c||c| c||c|c|}
 \hline
n & k &n&k&n&k \\
\hline
1 & 0 & 16 & 17 & 31 & 34 \\
2 & 0 & 17 & 18 & 32 & 35 \\
3 & 2 & 18 & 19 & 33 & 36 \\
4 & 4 & 19 & 20 & 34 & 37 \\
5 & 5 & 20 & 21 & 35 & 39 \\
6 & 6 & 21 & 22 & 36 & 40 \\
7 & 7 & 22 & 24 & 37 & 41 \\
8 & 8 & 23 & 25 & 38 & 42 \\
9 & 9 & 24 & 26 & 39 & 43 \\
10 & 10 & 25 & 27 & 40 & 44 \\
11 & 11 & 26 & 28 & 41 & 45 \\
12 & 12 & 27 & 29 & 42 & 47 \\
13 & 14 & 28 & 30 & 43 & 48 \\
14 & 15 & 29 & 32 & 44 & 49 \\
15 & 16 & 30 & 33 & 45 & 50 \\
 \hline
\end{tabular}
\end{center}
\caption{}
\label{table:2}
\end{table}

We know of no reason that the table  cannot  be  continued indefinitely.

\subsection{Eulerian Polynomials} The Eulerian numbers \cite{Eulerian} are given by $$A(n,m)=\sum_{k=0}^{m+1} (-1)^k\binom{n+1}{k}(m+1-k)^n$$ and the Eulerian polynomials \cite{Eulerian} by $$A_n(x) = \sum_{m=0}^{n-1} A(n,m)x^m.$$
Direct computer calculations for $n$ from  $0$ to $9$ show that $A_n(x)$ is completely real. But for $n > 9$ for those  $n$ tested, $A_n(x) $ has a few $i$-sections which have non-real roots. However, as mentioned in Stanley \cite{RPS}, $A_n(x)$ has only real roots. Note that from \cite{Eulerian} the Eulerian polynomials $A_n(x)$ are connected to the generating function of the sequence $a_i = (i+1)^k$ via
$$\sum_{i=0}^{\infty}(i+1)^kx^i= \frac{A_k(x)}{(1-x)^{k+1}},$$
which may help to explain some of the behavior in Section~\ref{kpower}.

\subsection{A Recursive Construction for Completely Real Polynomials}

The following result provides a way of constructing a completely real polynomial of arbitrary degree.

\begin{theorem}\label{distinctroots}
If $p(x)=\sum_{i=0}^n a_ix^i \in \mathbb{R}[x]$ is of degree $n\geq 1$ with all its roots real and distinct,  then there exists a nonzero $u \in \mathbb{R}$ such that $q(x)=p(x)+ux^{n+1}$ has all its roots real and distinct.
\end{theorem}
\begin{proof}
Let $r_1<r_2<\cdots<r_n$ denote the roots of $p(x)$. Let $s_1 \in (-\infty,r_1)$, $s_i\in (r_{i-1},r_i)$ for $2 \leq i \leq n$ and $s_{n+1}\in(r_n,\infty)$.  Note that $p(s_i)\neq 0$ for all $i$.  Let $u$ be small enough (in absolute value) so that $|us_{i}^{n+1}|<|p(s_i)|$ for $ 1 \leq i \leq n+1$, which implies $q(s_i)=p(s_i)+us_i^{n+1}$ is nonzero and has the same sign as $p(s_i)$.  Therefore, the signs of the $q(s_i)$ for $i \in [n+1]$ alternate since the signs of the $p(s_i)$ do so (as the roots of $p$ are distinct).  This implies $q$ has a root in each interval $(s_i,s_{i+1})$ for $1\leq i \leq n$.  Note that $q(s_1)$ and $q(s_{n+1})$ have the same sign if $n$ is even and opposite signs if $n$ is odd.  So if $n$ is even, then there must exist an $(n+1)$-st real root of $q$ in either the interval $(-\infty,s_1)$ or $(s_{n+1},\infty)$ since $\lim_{x\rightarrow -\infty}(q(x))$ and $\lim_{x\rightarrow \infty}(q(x))$ have opposite signs in this case, which implies $q$ has all its roots real and distinct.  If $n$ is odd, then the same conclusion is reached since $q(s_1)$ and $q(s_{n+1})$ have opposite signs in this case with $\lim_{x\rightarrow -\infty}(q(x))$ and $\lim_{x\rightarrow \infty}(q(x))$ of the same sign.
\end{proof}

\remark{} If $q$ has real distinct roots for some nonzero $u \in \mathbb{R}$, then the same is seen to hold for all $v$ in the interval $[-|u|,|u|]$. \medskip

We wonder if it is possible to extend Theorem \ref{distinctroots} to the case where $p$ has all its roots real, but with repeated roots allowed wherein we are seeking nonzero $u$ such that $q$ has all of its roots real.  The following result would require that the multiplicity of each root of $p$ be at most two.

\begin{proposition}
Let $u$ be a nonzero real number and $n \geq 3$.  Then the polynomial $q(x)=(x-1)^n+ux^{n+1}$ cannot have all of its roots real.
\end{proposition}
\begin{proof}
First note that if $r$ is a multiple root of $q$ iff $q(r)=q'(r)=0$.  Solving for $u$ and $r$ simultaneously then gives $u=-\frac{n^n}{(n+1)^{n+1}}$ and $r=n+1$.  Further, we have $q''(n+1)\neq 0$ (as $q''(n+1)=0$ is seen to be impossible), which implies the multiplicity of $r=n+1$ is exactly two.  So assume $u\neq-\frac{n^n}{(n+1)^{n+1}}$, in which case the roots of $q$ are all distinct.

We first rule out $u>0$.  If $n$ is odd, then $q$ has one negative root, by Descartes' rule.  Further, $q'(x)=n(x-1)^{n-1}+(n+1)ux^n>0$ for $x>0$, with $q(0)=-1<0$ and $\lim_{x\rightarrow\infty}(q(x))=\infty$, implies $q$ has only one positive root.  Thus $q$ has two real roots altogether since $u>0$ implies $q$ has no repeated roots.  If $n$ is even and $u>0$, then $q(x)>0$ for $x \geq 0$.  If $x<0$, then $q(-x)$ has one sign change and hence one negative root.  Thus $q$ has only one real root when $n$ is even and $u>0$.

Assume now $u<0$ with $u\neq -\frac{n^n}{(n+1)^{n+1}}$.  If $n$ is odd, then $q(x)<0$ for $x \leq 1$, so we consider possible roots $r$ on the interval $(1,\infty)$.  Note that $q(r)=0$ where $r>1$ iff $h(r)=1$, where
$$h(x)=\frac{f(x)}{g(x)} \text{ for } x>1, \text{ with } f(x)=\left(1-\frac{1}{x}\right)^n \text{ and } g(x)=-ux.$$
Upon considering $\ln(h(x))$, we have that $h(x)$ is increasing on $(1,n+1)$ and decreasing on $(n+1,\infty)$, with $h(1)=0=\lim_{x\rightarrow \infty}(h(x))$.  Thus, the equation $h(x)=1$ is satisfied by either zero or two numbers in $(1,\infty)$, for if $h(x)=1$ has exactly one solution $t$, then it must be that $t=n+1$ in which case $u=-\frac{n^n}{(n+1)^{n+1}}$ contrary to assumption.  This then implies $q$ has zero or two (distinct) roots on $(1,\infty)$, and hence zero or two real roots.

On the other hand, if $n$ is even, then $q(x)>0$ if $x \leq 0$ for $u<0$.  Consider once again the function $h(x)$.  Note that $h(x)$ is decreasing on $(0,1)$, increasing on $(1,n+1)$ and decreasing on $(n+1,\infty)$.  Then $\lim_{x \rightarrow 0^+}(h(x))=\infty$ and $h(1)=0$ implies $q$ has a single root on $(0,1)$.  Thus, reasoning as before, we have that $q$ has one or three real roots altogether (counting multiplicities) when $n$ is even.  Finally, if $u=-\frac{n^n}{(n+1)^{n+1}}$, then $r=n+1$ is a double root and it is seen that $q$ has two or three real roots  altogether (counting multiplicities) depending on if $n$ is odd or even.
\end{proof}

\remark{} Note that the prior result is not true for $n=2$, i.e., there exists nonzero $u$ (all lying in the interval $(-1/6,0)$) for which $q(x)=(x-1)^2+ux^3$ has all its roots real.  Based on numerical evidence, it appears that such $u$ can be found in cases when $p(x)$ has each of its roots real with multiplicity one or two, though we do not have a complete proof.

\section{Proofs for Sequences from \ref{binom}--\ref{5.7} and Others that ${b_n = n \bmod 2}$\label{proofs}}

Our first result below provides necessary and sufficient conditions for $b_n$ to be trivial in the case when $a_n$ is a quadratic polynomial in $n$ and accounts for the behavior witnessed above in subsections \ref{5.5} and \ref{5.6}.

\begin{theorem}\label{quadraticseq}
Let $a_n=an^2+bn+c$, where $a,b, c \geq 0$ are real numbers.  Then in order for $b_n$ to be the trivial sequence, it is necessary and sufficient that one of the following conditions be satisfied:
\begin{align*}
(\text{I})&~b>a \geq 0  \text{ and } b/2 \leq c <b-a, \\
(\text{II})&~b>a \geq 0,~c=b-a \text{ and } b \geq 2a,\\
(\text{III})&~c> \max\{0,b-a\}\text{ and either } (\text{i})~c \geq \frac{a+b}{2},~(\text{ii})~\frac{b}{2}\leq c \leq \frac{3b-a}{4} \text{ or }\\
&~(\text{iii})~\frac{3b-a}{4}<c<\frac{a+b}{2},~8ac \geq (a+b)^2 \text{ and the polynomial }
q_n(x) \text{ defined}\\
 &~\text{below has one root in } (-1,0) \text{ for all } n\geq1 \text{ odd.}
\end{align*}
\end{theorem}
\begin{proof}
Let $p_n(x)=\sum_{i=0}^na_ix^i$ for $n \geq 0$. We wish to find conditions on $a,b, c\geq 0$ under which $p_n(x)$ has zero or one real roots depending on whether $n$ is even or odd for all $n$. Note that $a=b=c=0$ is clearly disallowed, whereas if $c=0$ and $a$ and $b$ are not both zero, then $x=0$ would be a real root in the even case. Thus, we must have $c>0$.  Further, the $n=0$ and $n=1$ cases are self-evident and follow from $c>0$, and so we may assume $n \geq 2$ for the remainder of the proof.   A routine calculation, using the fact $\sum_{i=0}^nx^i=\frac{1-x^{n+1}}{1-x}$ and differentiation, implies
$$p_n(x)=\frac{q_n(x)}{(1-x)^3},$$
where
\begin{align*}
q_n(x)=&-(an^2+bn+c)x^{n+3}+(a(2n^2+2n-1)+b(2n+1)+2c)x^{n+2}\\
&-(a(n+1)^2+b(n+1)+c)x^{n+1}+(a-b+c)x^2+(a+b-2c)x+c.
\end{align*}

Since $a,b, c\geq 0$ with $c>0$, the polynomial $p_n(x)$ has no non-negative roots, so we need only check for negative roots.  Further, $p_n(x)$ and $q_n(x)$ have the same negative roots, so we consider the latter.  Thus, we seek conditions under which $q_n(x)$ has zero or one negative roots depending on the parity of $n$ for all $n$.  To do so, we apply Descartes' rule of signs and count the sign changes in $q_n(-x)$. Note that $n \geq 2$ implies the terms in $q_n(x)$ are all distinct.  Since $q_n(-1)>0$ and $q_n'(x)<0$ for $x<-1$ if $n\geq 2$ is even, we may restrict attention to the interval $(-1,0)$ when considering roots of $q_n(x)$ in the even case.  Further, the same may be done in the odd case since $q_n(-1)<0$ with $q_n'(x)>0$ if $x<-1$ and $n\geq 3$ is odd.

First suppose $b>a \geq 0$ and $c<b-a$.  Then the $x^2$ coefficient in $q_n(x)$ is negative and $q_n(-x)$ has one sign change if $n$ is odd.  Then $q_n(x)$, and hence $p_n(x)$, has one negative root, implying $p_n(x)$ has one real root in this case.  However, when $n$ is even, we have that $q_n(-x)$ has two sign changes and hence $q_n(x)$ has either zero or two negative roots.  We now determine a condition on $a,b,c$ which ensures the latter cannot occur. First let
$$h(x)=(a-b+c)x^2+(a+b-2c)x+c.$$
In order to guarantee $q_n(x)>0$ on $(-1,0)$ when $n$ is even, we must have $h(x) \geq 0$ on the interval.
For if not and $h(x_0)<0$ for some $x_0\in(-1,0)$, then $q_n(x_0)<0$ upon choosing $n$ even sufficiently large, since one can make the positive quantity $q_n(x_0)-h(x_0)$ as small as one pleases by so doing.  Conversely, if $q_m(x_0)\leq 0$ for some $m$ even, then $h(x_0)<0$.

To show $h(x) \geq 0$ on $(-1,0)$, we consider the $x$-coordinate, which we will denote by $q$, of the vertex of $h$, which corresponds to a local maximum in this case since $a-b+c<0$.  Note $q=-\frac{a+b-2c}{2(a-b+c)}$ and we consider cases based on $q$.  If $q \geq -1$, then one needs to verify $h(0) \geq 0$ and $h(-1) \geq 0$ to ensure $h(x) \geq 0$ on $(-1,0)$.  Note that $h(-1) \geq 0$ if and only if $b \leq 2c$, with $h(0)=c>0$.  On the other hand, if $q <-1$, then we have $3b-a<4c$ and again $b \leq 2c$ since $b-a>0$.  Conversely, if $b \leq 2c$, then we have $h(x) \geq 0$ for all $-1<x<0$, which completes case (I).

We now proceed with case (II) and assume $b>a \geq 0$ and $c=b-a$.  Then the coefficient of $x^2$ in $q_n(x)$ vanishes and there is one sign change in $q_n(-x)$ if $n$ is odd.  If $n$ is even, then again the problem reduces to showing $h(x)=(a+b-2c)x+c \geq 0$ on $(-1,0)$.  If $a+b-2c=a+b-2(b-a) \leq 0$, i.e., $b \geq 3a$, this clearly holds.  If $a+b-2c>0$, then it suffices to verify $h(-1) \geq 0$, which holds if and only if $b \geq 2a$.  Thus, if $c=b-a>0$, the necessary condition for $p_n(x)$ having the minimum number of real roots for all $n$ is $b \geq 2a$, which is seen to be sufficient as well.

Finally, for case (III), suppose $c>\max\{0,b-a\}$.  Then $c \leq \frac{3b-a}{4}$ in subcase (ii) implies $c<\frac{a+b}{2}$, whence (i), (ii) and (iii) are mutually disjoint.  To see this, note that $c \leq \frac{3b-a}{4}$ iff $\frac{a+b-2c}{2(a-b+c)} \geq 1$ since $a-b+c>0$.  Thus, $\frac{a+b-2c}{2(a-b+c)}>0$, which implies $c \leq \frac{a+b}{2}$.  One may verify that $q_n(-x)$ has zero or one sign changes depending on whether $n$ is even or odd iff $c \geq \frac{a+b}{2}$, which gives (i).  So assume $c < \frac{a+b}{2}$.  Then there are three sign changes in $q_n(-x)$ if $n$ is odd and two if $n$ is even.  We consider further subcases based on $q$; note $c<\frac{a+b}{2}$ implies $q<0$. Further $q \leq -1$ iff $c \leq \frac{3b-a}{4}$.

So first assume $q \leq -1$.  If $n$ is odd, then $q_n'(x)>0$ for $1 \leq x <0$  since $q_n(x)$ is the sum of $q_n(x)-h(x)$ and $h(x)$, which are both increasing on the interval.  Therefore $q_n'(x)>0$ for all $x<0$ in this case, which implies a single negative root since $\lim_{x\rightarrow -\infty}(q_n(x))=-\infty$ and $q_n(0)>0$.  On the other hand, if $n$ is even, then we must require $h(-1) \geq 0$ , i.e., $b \leq 2c$, in order to guarantee that $q_n(x)$ has no roots on $(-1,0)$.  Thus, the required condition when $q \leq -1$ in (III) is $\frac{b}{2}\leq c \leq \frac{3b-a}{4}$, which is also seen to be sufficient and yields (ii).  Finally, if $-1<q<0$, then we must have $\frac{3b-a}{4}<c<\frac{a+b}{2}$ as well as the requirement concerning $q_n(x)$ for $n$ odd given in part (iii) of (III). If $n$ is even, then we require further $h(q) \geq 0$, which occurs iff $8ac \geq (a+b)^2$ as one may verify.  This completes the cases in (III) and the proof.
\end{proof}

\remark{} The sequences $a_n=n+1$ and $a_n=(n+1)^2$ from subsections \ref{5.5} and \ref{5.6} above correspond respectively to the $a=0,b=c=1$ and $a=c=1,b=2$ cases of Theorem \ref{quadraticseq} and are covered under (II).\medskip

Our next result simultaneously accounts for the behavior observed in subsections \ref{binom} and \ref{5.7}.

\begin{theorem}\label{Catalangen}
Suppose $a_n$ is a sequence of positive real numbers satisfying $\frac{a_{n+1}}{a_n}=\frac{an+b}{cn+d}$ for $n \geq 0$, where $a,b,c,d\geq 0$ with $c \leq d$. Assume further that $a,b$ are not both zero and the same for $c,d$. Then $b_n$ is the trivial sequence for all such $a_n$.
\end{theorem}
\begin{proof}
Let $p_n(x)=\sum_{i=0}^na_ix^i$ for $n \geq 0$.  First suppose $c>0$.  Upon dividing by $c$ and renaming variables, one may assume a ratio of the form $\frac{a_{n+1}}{a_n}=\frac{an+b}{n+c}$, where $a,b \geq 0$ are not both zero and $c \geq 1$. Multiplying both sides of $(i+c)a_{i+1}=(ai+b)a_i$ by $x^i$, and summing over $0 \leq i \leq n$, gives
$$\sum_{i=0}^n(i+c)a_{i+1}x^i=\sum_{i=0}^n(ai+b)a_ix^i.$$
The left side of the last equality may be written as
\begin{align*}
\sum_{i=1}^{n+1}(i+c-1)a_ix^{i-1}&=(n+c)a_{n+1}x^n+\sum_{i=1}^n(i+c-1)a_ix^{i-1}\\
&=(n+c)a_{n+1}x^n+\sum_{i=1}^nia_ix^{i-1}+(c-1)\sum_{i=1}^na_ix^{i-1}\\
&=(n+c)a_{n+1}x^n+p_n'(x)+\frac{c-1}{x}(p_n(x)-a_0),
\end{align*}
and the right as
$$\sum_{i=0}^n(ai+b)a_ix^i=ax\sum_{i=1}^nia_ix^{i-1}+b\sum_{i=0}^na_ix^i=axp_n'(x)+bp_n(x).$$
Therefore, we get
$$(n+c)a_{n+1}x^n+p_n'(x)+\frac{c-1}{x}(p_n(x)-a_0)=axp_n'(x)+bp_n(x),$$
which may be rewritten as
\begin{equation}\label{dife1}
p_n'(x)+\frac{c-1-bx}{x(1-ax)}p_n(x)=\frac{(c-1)a_0-(n+c)a_{n+1}x^{n+1}}{x(1-ax)}.
\end{equation}

First assume $a>0$.  Note that by partial fractions, we have
\begin{align*}
\int\frac{c-1-bx}{x(1-ax)}dx&=\int\left(\frac{a(c-1)-b}{1-ax}+\frac{c-1}{x}\right)dx\\
&=\frac{b-a(c-1)}{a}\ln|1-ax|+(c-1)\ln|x|.
\end{align*}
Proceeding as one would in solving a first-order linear differential equation, we multiply both sides of \eqref{dife1} by the integrating factor
$$\exp\left(\int\frac{c-1-bx}{x(1-ax)}dx\right)=|1-ax|^{\frac{b-a(c-1)}{a}}|x|^{c-1},$$
which implies
\begin{align}
&\frac{d}{dx}\left((1-ax)^{\frac{b-a(c-1)}{a}}|x|^{c-1}p_n(x)\right)\notag\\
&\qquad\qquad\quad=\left(\frac{(c-1)a_0-(n+c)a_{n+1}x^{n+1}}{x}\right)(1-ax)^{\frac{b-ac}{a}}|x|^{c-1}, \quad x<0.\label{dife2}
\end{align}

Again, we need only consider possible negative roots of $p_n(x)$ since $a_n>0$ for all $n \geq 0$ implies $p_n(x)>0$ for all $x \geq 0$.  Let
$$f_n(x)=(1-ax)^{\frac{b-a(c-1)}{a}}|x|^{c-1}p_n(x), \qquad n \geq 0.$$
Note that $p_n(x)$ and $f_n(x)$ have the same negative roots, and we find it more convenient to deal with the roots of the latter. We consider cases on $c$.  First suppose $c>1$.  Then by \eqref{dife2}, we have $f_n'(x)<0$ for $x<0$ if $n$ is even, since $(c-1)a_0-(n+c)a_{n+1}x^{n+1}>0$ for all such $x$ and $n$.  Since $c>1$ implies $f_n(0)=0$, it follows that $f_n$, and hence $p_n$, has no negative roots.  On the other hand, if $n$ is odd, then $c>1$ implies $f_n'(x)$ for $x<0$ has exactly one sign change, from positive to negative.  Since $f_n(0)=0$ and $\lim_{x\rightarrow-\infty}(f_n(x))=-\infty$ for $n$ odd, it follows that $f_n(x)$, and hence $p_n(x)$, has a single negative root, which completes the case when $c>1$.  On the other hand, if $c=1$, then \eqref{dife2} gives $f_n'(x)<0$ for $x<0$ if $n$ is even and $f_n'(x)>0$ for $x<0$ if $n$ is odd.  Since $c=1$ implies $f_n(0)=p_n(0)=a_0>0$, the same conclusion as before is reached concerning the negative roots of $f_n$, and hence $p_n$.

A similar argument applies if $a=0$, where instead of \eqref{dife2}, one gets
\begin{align*}
&\frac{d}{dx}\left(e^{-bx}|x|^{c-1}p_n(x)\right)\notag\\
&\qquad\qquad\quad=\left(\frac{(c-1)a_0-(n+c))a_{n+1}x^{n+1}}{x}\right)e^{-bx}|x|^{c-1}, \quad x<0.
\end{align*}
Considering the cases $c>1$ and $c=1$ as before leads to the same conclusion.

Now suppose $c=0$ in the original problem above where $\frac{a_{n+1}}{a_n}=\frac{an+b}{cn+d}$.  Since $d>0$, we then may write $\frac{a_{n+1}}{a_n}=pn+q$, where $p,q\geq 0$ are not both zero.  If $p=0$, then the result is clear since one gets in this case an exponential sequence for $a_n$ as in subsection \ref{expon} above.  So assume $p>0$.  If we once again denote $\sum_{i=0}^na_ix^i$ by $p_n(x)$, then we get the differential equation
$$p'_n(x)+\frac{qx-1}{px^2}p_n(x)=\frac{a_{n+1}x^{n+1}-a_0}{px^2},$$
which may be rewritten as
$$g'_n(x)=|x|^{q/p-2}e^{1/(px)}\left(\frac{a_{n+1}x^{n+1}-a_0}{p}\right), \qquad x<0,$$
where
$$g_n(x)=|x|^{q/p}e^{1/(px)}p_n(x).$$
Then we may proceed as in the argument above to reach the same conclusion, upon noting
$$g_n(0)=\lim_{x\rightarrow 0^-}(g_n(x))=p_n(0)\cdot\lim_{x\rightarrow 0^-}\left(|x|^{q/p}e^{1/(px)}\right)=a_0\cdot 0=0,$$
which completes the proof.
\end{proof}

\remark{} The sequences $a_n=\binom{n+k}{k}$ and $a_n=\frac{1}{n+1}\binom{2n}{n}$ from subsections \ref{binom} and \ref{5.7} correspond to the cases of Theorem \ref{Catalangen} when $a=c=d=1,b=k+1$ and $a=4,b=d=2,c=1$, respectively, which accounts for the triviality of $b_n$ witnessed above by these sequences.  Further, we note that the central binomial sequences $a_n=\binom{2n}{n}$ and $a_n=\binom{2n+1}{n}$ for $n \geq 0$ correspond to the $a=4,b=2,c=d=1$ and $a=4,b=6,c=1,d=2$ cases of Theorem \ref{Catalangen}, and hence they too have trivial $b_n$. Taking $a=b=d=1,c=0$ gives the same result for $a_n=n!$.  If $a=2,b=d=1,c=0$ or $a=b=2,c=0,d=1$, then one obtains the result for the double factorial sequences $a_n=(2n-1)!!$ and $a_n=(2n)!!$, respectively.\medskip

The following result provides a way of constructing sequences $a_n$ for which $b_n$ is minimal.

\begin{theorem}
Suppose $p(x)=\sum_{i=0}^na_ix^i \in \mathbb{R}[x]$ is of degree $n$ with $a_0 >0$ and that $p(x)$ has zero or one real roots depending on if $n$ is even or odd.  Then there exists $u>0$ such that $q(x)=p(x)+ux^{n+1}$ has one or zero real roots depending on the parity of $n$.
\end{theorem}
\begin{proof}
First assume $n$ is odd.  Note that the condition $a_0>0$ is necessary for the existence of such a $u>0$, for if not, and $a_0 \leq 0$, then $q(x)$ would have at least one real root for all $u>0$.  By continuity, one may pick $\delta>0$ such that $p(x)>\frac{a_0}{2}$ for all $x \in (-\delta,\delta)$.  Note that the real root of $p(x)$ has sign opposite that of the leading coefficient $a_n$.  First suppose $a_n>0$.  Let $h(x)=-\sum_{i=0}^na_ix^{i-n-1}$ for $(-\infty,-\delta]$.  Note that $h(x)$ is bounded on $(-\infty,-\delta]$, being the sum of $n+1$ bounded functions on the interval. Let $M>0$ such that $h(x)<M$ on $(-\infty,-\delta]$.   If $u\geq M$, then we have $ux^{n+1}>-\sum_{i=0}^na_ix^i$ for all $-\infty<x\leq -\delta$, i.e., $q(x)>0$ on the interval for such $u$.  If $-\delta<x<\infty$, then already $p(x)>0$ on this interval since $p$ only has one real root, which is negative as $a_n>0$ by assumption.  Hence, $q(x)\geq p(x)>0$ for $-\delta<x<\infty$, implying $q(x)>0$ for all $x$ as desired.  A similar argument applies if $a_n<0$, upon picking $u>\max_{x\in[\delta,\infty)}(h(x))$.

Now suppose $n$ is even, where clearly we may assume $n \geq 2$.  Then $a_0>0$ implies we must find $u>0$ such that $q$ has exactly one negative root. Such a $u$ is the one sought in the theorem statement since $q(x)\geq p(x)>0$ for $x \geq 0$.  To find $u$, again let  $\delta>0$ such that $p(x)>\frac{a_0}{2}$ for all $x \in (-\delta,\delta)$.  Let $m=\min_{x\in[-\delta,0]}(p(x))$; note that $\frac{a_0}{2}\leq m \leq a_0$.  We require $u<\frac{m}{\delta^{n+1}}$, for then $ux^{2n+1}+m>0$ for $-\delta \leq x \leq 0$ would imply $q(x)=ux^{n+1}+p(x)>0$ on the interval for such $u$.  Let us require further $q'(x)>0$ for all $x \in (-\infty,-\delta)$.  Note $q'(x)>0$ iff $u(n+1)x^n+\sum_{i=1}^nia_ix^{i-1}>0$ iff $u>j(x)$, where $j(x)=-\frac{1}{n+1}\sum_{i=1}^nia_ix^{i-n-1}$.  Note that $j(x)$ is a bounded function on $(-\infty,-\alpha]$ for any $\alpha>0$.  If $\alpha>0$ is chosen small enough, then $|j(x)|$ achieves its maximum value on $(-\infty,-\alpha]$ at $x=-\alpha$; to see this, note that the dominant term in $j(-\alpha)$ for small $\alpha$ occurs when $i=1$ in the sum, assuming $a_1 \neq 0$.  Since the $i=1$ term is of the form $\frac{k}{\alpha^{n}}$ for a nonzero constant $k$, it follows that $|j(-\alpha)|<\frac{m}{\alpha^{n+1}}$ (*) for $\alpha$ sufficiently small.  Thus, we let $\alpha$, where $0<\alpha<\delta$, be small enough such that both $|j(x)|$ achieves its maximum value on $(-\infty,-\alpha]$ at $x=-\alpha$ and the inequality (*) holds.

Now pick $u>0$ so that $|j(-\alpha)|<u<\frac{m}{\alpha^{n+1}}$, where $\alpha$ is as specified above.  For such $u$, one has $u>\max_{x\in(-\infty,-\alpha]}(j(x))$, i.e., $(n+1)ux^n+\sum_{i=1}^nia_ix^{i-1}>0$, or $q'(x)>0$ for all $-\infty<x\leq -\alpha$.  Also, $u<\frac{m}{\alpha^{n+1}}$ and $n$ even implies $q(x)=ux^{n+1}+p(x)\geq u(-\alpha)^{n+1}+p(x) \geq u(-\alpha)^{n+1}+m>0$ for all $x \in [-\alpha,0)$.  Since $q'(x)>0$ for $-\infty<x<-\alpha$, $q(x)>0$ for $-\alpha \leq x<0$ and $\lim_{x \rightarrow -\infty}(q(x))=-\infty$, it follows that $q(x)$ has exactly one negative root as desired, which completes the case when $a_1 \neq 0$.  On the other hand, if $a_1=0$, consider the smallest index $1<i \leq n$ such that $a_i \neq 0$. Again, we have $|j(x)|$ achieving its maximum value on $(-\infty,-\alpha]$ at $x=-\alpha$ if $\alpha>0$ is sufficiently small.  If $\alpha<\delta$ is also small enough so as to ensure (*) holds, then once again pick $u$ strictly between the quantities compared in (*) and complete the proof as before. Thus, one can always find  $u>0$ such that $q(x)$ has a single negative root, which completes the even case and the proof.
\end{proof}

We conclude with the following result for the minimality of $b_n$ when $a_n=n^n$ since the method of comparing the various terms in its proof can be applied to other sequences with a sufficiently fast growth rate.

\begin{theorem}\label{n^nthm}
The polynomial $p_m(x)=1+\sum_{i=1}^m(ix)^i$ for $m \geq 0$ has zero or one real roots depending on if $m$ is even or odd.
\end{theorem}
\begin{proof}
First suppose $m=2n$ where $n \geq 1$ and we show that $p_m(x)>0$ for all $x$.  Clearly, one may restrict attention to the $x<0$ case.  Upon replacing $x$ by $-x$, we must show
\begin{equation}\label{n^nthme1}
\sum_{i=1}^n((2i-1)x)^{2i-1}<1+\sum_{i=1}^n (2ix)^{2i}, \qquad x>0.
\end{equation}
Let $u_n=\frac{1}{2n}\left(1-\frac{1}{2n}\right)^{2n-1}$; note that $u_n$ is a decreasing function of $n$ as one may verify.  Comparing the $i$-th summands where $1 \leq i \leq n$ on both sides of \eqref{n^nthme1}, note that $(2ix)^{2i}>((2i-1)x)^{2i-1}$ iff $x>u_i$.  Thus \eqref{n^nthme1} clearly holds, by addition, if $x>u_1$.  Otherwise, if $x \leq u_1$, then let $\ell$ be the uniquely determined index in $[n]$ such that $u_{\ell+1}<x \leq u_\ell$, where $u_{n+1}=0$.  Then, by the first part of Lemma \ref{lem1} below, we have
$$\sum_{i=1}^\ell((2i-1)x)^{2i-1}\leq \sum_{i=1}^\ell((2i-1)u_\ell)^{2i-1}<1,$$
with
$$((2i-1)x)^{2i-1}<(2ix)^{2i}, \qquad \ell+1 \leq i \leq n,$$
where the latter inequalities hold since $x>u_{\ell+1}>\cdots>u_n$.  Then \eqref{n^nthme1} follows from addition of the preceding set of inequalities, which completes the even case.

Now suppose $m=2n+1$ where $n \geq 1$ and it suffices to show $p_m'(x)>0$ for all $x$, where clearly we may restrict attention to the case $x<0$.  Replacing $x$ by $-x$, we must prove
\begin{equation}\label{n^nthme2}
\sum_{i=1}^n(2i)^{2i+1}x^{2i-1}<1+\sum_{i=1}^n(2i+1)^{2i+2}x^{2i}, \qquad x>0.
\end{equation}
Let $v_n=\frac{1}{2n+1}\left(1-\frac{1}{2n+1}\right)^{2n+1}$.  Note that $v_n$ is decreasing and that $(2i+1)^{2i+2}x^{2i}>(2i)^{2i+1}x^{2i-1}$ iff $x>v_i$.  Thus \eqref{n^nthme2} can be established in the same way as was \eqref{n^nthme1} above using the second part of Lemma \ref{lem1}, upon considering the cases $x>v_1$ or $v_{\ell+1}<x\leq v_\ell$ for some $\ell \in [n]$.  This completes the odd case and the proof.
\end{proof}

\begin{lemma}\label{lem1}
Let $u_n=\frac{1}{2n}\left(1-\frac{1}{2n}\right)^{2n-1}$ and $v_n=\frac{1}{2n+1}\left(1-\frac{1}{2n+1}\right)^{2n+1}$. Then for $n \geq 1$, we have $\sum_{i=1}^n(2i-1)^{2i-1}u_n^{2i-1}<1$ and $\sum_{i=1}^n(2i)^{2i+1}v_n^{2i-1}<1$.
\end{lemma}
\begin{proof}
To show $\sum_{i=1}^n(2i-1)^{2i-1}u_n^{2i-1}<1$, i.e.,
$$\sum_{i=1}^n\left(\frac{2i-1}{2n}\right)^{2i-1}\left(1-\frac{1}{2n}\right)^{(2n-1)(2i-1)}<1,$$
it suffices to prove
\begin{equation}\label{leme0}
\sum_{i=1}^nf(i)<1, \qquad n \geq 1,
\end{equation}
where $f(i)=\left(\frac{2i-1}{2n-1}\right)^{2i-1}e^{-(2i-1)}$ for $1 \leq i \leq n$, upon noting $\{\left(1-\frac{1}{n}\right)^n\}_{n\geq1}$ is increasing with limit $\frac{1}{e}$. Thus, we seek to maximize the function
$$g(x)=\left(\frac{x}{2n-1}\right)^xe^{-x}, \qquad 1 \leq x \leq 2n-1,$$
where $n\geq 1$ is fixed.  Upon considering the log, we have that $g(x)$ is decreasing on the interval $(1,2n-1)$ and hence $\sum_{i=1}^nf(i)\leq nf(1)<1$, which implies \eqref{leme0}.

The $n=1,2$ cases of the second inequality may be verified directly, so assume $n \geq 3$. To show $\sum_{i=1}^n(2i)^{2i+1}v_n^{2i-1}<1$,
i.e.,
$$\frac{1}{2n+1}\sum_{i=1}^n(2i)^3\left(\frac{2i}{2n+1}\right)^{2i-2}\left(1-\frac{1}{2n+1}\right)^{(2n+1)(2i-1)}<1,$$
it suffices to prove
\begin{equation}\label{leme1}
\frac{1}{2n+1}\sum_{i=1}^nh(i)<1, \qquad n \geq 3,
\end{equation}
where $h(i)=(2i)^3\left(\frac{2i}{2n+1}\right)^{2i-2}e^{-(2i-1)}$.
For \eqref{leme1}, it suffices to show
\begin{equation}\label{leme2}
h(i) < 2, \qquad 2 \leq i \leq n,
\end{equation}
for then we would have
$$\sum_{i=1}^nh(i)=\frac{8}{e}+\sum_{i=2}^nh(i)<3+2(n-1)=2n+1,$$
as desired.

We then seek to maximize the function
$$j(x)=x^3\left(\frac{x}{2n+1}\right)^{x-2}e^{1-x}, \qquad 4 \leq x \leq 2n,$$
where $n \geq 3$ is fixed.  Considering $\ln(j(x))$ and its first two derivatives, it is seen that $j(x)$ is decreasing and then increasing on the interval $(4,2n)$.  Therefore, in order to maximize $j(x)$ on $[4,2n]$, we need only consider the values at the endpoints $x=4$ and $x=2n$.  Thus, for \eqref{leme2}, it suffices to show $\max\{h(2),h(n)\}<2$.
One may verify $$h(2)=\frac{64}{e^3}\left(\frac{4}{2n+1}\right)^2<2 \text{ and } h(n)=(2n)^3\left(\frac{2n}{2n+1}\right)^{2n-2}e^{1-2n}<2$$ for all $n \geq 3$, which implies \eqref{leme2} and completes the proof.
\end{proof}

The method used to prove the last theorem is potentially applicable to other sequences of positive real numbers, especially hyper-exponential sequences.  To show that the polynomial $\sum_{i = 0}^ma_ix^i$ where $m \geq 0$ and $a_i>0$ for all $i$ has the minimal number of real zeros for all $m$, it is enough to demonstrate the following:
\begin{align*}
(\text{i})&~\text{The sequences } u_n=\frac{a_{2n-1}}{a_{2n}} \text{ and } v_n=\frac{2na_{2n}}{
(2n+1)a_{2n+1}} \text{ are decreasing in }n,\\
(\text{ii})&~\sum_{i=1}^{n}a_{2i-1}u_n^{2i-1}<a_0, \quad n \geq 1,\\
(\text{iii})&~\sum_{i=1}^n2ia_{2i}v_n^{2i-1}<a_1, \quad n \geq 1.
\end{align*}
Once (i)--(iii) have been shown, a proof similar to that of Theorem \ref{n^nthm} above can be given for the sequence $a_n$.  For example, it is possible to show (we omit the details) steps (i)--(iii) in the case when $a_n=2^{n^k}$ for $n \geq 0$, where $k \geq 2$ is fixed.  The conditions (i)--(iii) above are by no means necessary; for example, $a_n=2^n$ has minimal $b_n$ sequence (see 5.3 above) despite it failing to satisfy any of (i)--(iii).

\section{Miscellaneous Examples}
Here, we exhibit the $b$ sequences corresponding  to some classical  number theoretic functions $a$.

\subsection{$a_n=\phi(n)$ (Euler totient function)}
Computations suggest that  $b_n = n \bmod 2$ and we seek a proof or counterexample.

\subsection{$a_n=\pi(n +2)$, the number of primes $\leq n+2$}
In this case $b$ begins:
$$
\begin{array}{cccccccccccccccccccc}
0, & 1, & 0, & 1, & 0, & 1, & 0, & 1, & 2, & 1, & 2, & 1, & 2, & 1, & 2,\\
1, & 2, & 1, & 2, & 1, & 2, & 1, & 2, & 1, & 2, & 1, & 2, & 1, & 2, & 1,\\
2, & 1, & 2, & 1, & 2, & 1, & 2, & 1, & 2, & 1, & 2, & 1, & 2, & 1, & 2,\\
1, & 2, & 1, & 2, & 1, & 2, & 1, & 2, & 1, & 2, & 1, & 2, & 1, & 2, & 1,\\
2, & 1, & 2, & 1, & 2, & 1, & 2, & 1, & 2, & 1, & 2, & 1, & 2, & 1, & 2,\\
1, & 2, & 1, & 2, & 1, & 2, & 1, & 2, & 1, & 2, & 1, & 2, & 1, & 2, & 1,\\
2, & 1, & 2, & 1, & 2, & 1, & 2, & 1, & 2, & 1, & 2, & 1, & 2, & 1, & 2,\\
1, & 2, & 1, & 2, & 1, & 2, & 1, & 2, & 1, & 2, & 1, & 2, & 1, & 2, & 1,\\
2, & 1, & 2, & 1, & 2, & 1, & 2, & 1, & 2, & 1, & 2, & 1, & 2, & 1, & 2,\\
1, & 2, & 1, & 2, & 1, & 2, & 1, & 2, & 1, & 2, & 1, & 2, & 1, & 2, & \dots

\end{array}
$$
\subsection{$a_n = \mu(n+1)$, Moebius Function}
In this case $b$ begins:
$$
\begin{array}{cccccccccccccccccccc}
0, & 1, & 2, & 2, & 2, & 3, & 2, & 2, & 2, & 3, & 2, & 2, & 2, & 3, & 4,\\
4, & 2, & 2, & 2, & 2, & 4, & 3, & 2, & 2, & 2, & 3, & 3, & 3, & 2, & 3,\\
2, & 2, & 4, & 3, & 4, & 4, & 4, & 3, & 4, & 4, & 4, & 3, & 2, & 2, & 2,\\
3, & 2, & 2, & 2, & 2, & 4, & 4, & 2, & 2, & 4, & 4, & 4, & 5, & 4, & 4,\\
2, & 3, & 3, & 3, & 4, & 5, & 2, & 2, & 4, & 3, & 4, & 4, & 2, & 3, & 3,\\
3, & 4, & 3, & 2, & 2, & 2, & 3, & 2, & 2, & 4, & 3, & 4, & 4, & 4, & 4,\\
4, & 4, & 4, & 5, & 4, & 4, & 4, & 4, & 4, & 4, & 4, & 5, & 2, & 2, & 2,\\
3, & 2, & 2, & 2, & 3, & 4, & 4, & 2, & 3, & 4, & 4, & 4, & 3, & 4, & 4,\\
4, & 3, & 4, & 4, & 4, & 4, & 4, & 4, & 4, & 5, & 4, & 4, & 4, & 5, & 5,\\
5, & 4, & 3, & 2, & 2, & 4, & 3, & 4, & 4, & 4, & 5, & 5, & 5, & 4, & \dots
\end{array}
$$

\subsection{$a_n = \sigma(n+1)$, Sum of  Divisors of $n+1$}
In this case $b$ begins:
$$
\begin{array}{cccccccccccccccccccc}
0, & 1, & 0, & 1, & 2, & 1, & 2, & 1, & 2, & 1, & 2, & 1, & 2, & 1, & 2,\\
1, & 2, & 1, & 2, & 1, & 2, & 1, & 2, & 1, & 2, & 1, & 2, & 1, & 2, & 1,\\
2, & 1, & 2, & 1, & 2, & 1, & 2, & 1, & 2, & 1, & 2, & 1, & 2, & 1, & 2,\\
1, & 2, & 1, & 2, & 1, & 2, & 1, & 2, & 1, & 2, & 1, & 2, & 1, & 2, & 1,\\
2, & 1, & 2, & 1, & 2, & 1, & 2, & 1, & 2, & 1, & 2, & 1, & 2, & 1, & 2,\\
1, & 2, & 1, & 2, & 1, & 2, & 1, & 2, & 1, & 2, & 1, & 2, & 1, & 2, & 1,\\
2, & 1, & 2, & 1, & 2, & 1, & 2, & 1, & 2, & 1, & 2, & 1, & 2, & 1, & 2,\\
1, & 2, & 1, & 2, & 1, & 2, & 1, & 2, & 1, & 2, & 1, & 2, & 1, & 2, & 1,\\
2, & 1, & 2, & 1, & 2, & 1, & 2, & 1, & 2, & 1, & 2, & 1, & 2, & 1, & 2,\\
1, & 2, & 1, & 2, & 1, & 2, & 1, & 2, & 1, & 2, & 1, & 2, & 1, & 2, & \dots
\end{array}
$$

\subsection{$a_n = \sigma_7(n+1)$, Sum of  the  $7$-th Powers of the  Divisors of $n+1$}
In this case $b$ begins:
$$
\begin{array}{cccccccccccccccccccc}
0, & 1, & 2, & 1, & 2, & 1, & 2, & 1, & 2, & 1, & 2, & 1, & 4, & 3, & 4,\\
3, & 4, & 3, & 4, & 3, & 4, & 3, & 4, & 3, & 4, & 3, & 4, & 3, & 4, & 3,\\
4, & 3, & 4, & 3, & 4, & 3, & 4, & 3, & 4, & 3, & 4, & 3, & 4, & 3, & 4,\\
3, & 4, & 3, & 4, & 3, & 4, & 3, & 4, & 3, & 4, & 3, & 4, & 3, & 4, & 3,\\
4, & 3, & 4, & 3, & 4, & 3, & 4, & 3, & 4, & 3, & 4, & 3, & 4, & 3, & 4,\\
3, & 4, & 3, & 4, & 3, & 4, & 3, & 4, & 3, & 4, & 3, & 4, & 3, & 4, & 3,\\
4, & 3, & 4, & 3, & 4, & 3, & 4, & 3, & 4, & 3, & 4, & 3, & 4, & 3, & 4,\\
3, & 4, & 3, & 4, & 3, & 4, & 3, & 4, & 3, & 4, & 3, & 4, & 3, & 4, & 3,\\
4, & 3, & 4, & 3, & 4, & 3, & 4, & 3, & 4, & 3, & 4, & 3, & 4, & 3, & 4,\\
3, & 4, & 3, & 4, & 3, & 4, & 3, & 4, & 3, & 4, & 3, & 4, & 3, & 4, & \dots
\end{array}
$$

\section{Acknowledgment} We wish to thank  Mourad Ismail for steering us to relevant papers in analysis and for  obtaining a copy of  the difficult to locate  paper by Ostrovskii \cite{O}. We also wish to thank Larry Dunning for providing us with a copy of the  paper by Kobel, Lovava and Sagraloff \cite{KRS}.

\end{document}